\documentclass[12pt, reqno]{amsart}
\usepackage[utf8]{inputenc}
\usepackage{amssymb,mathtools,cite,enumerate,color,eqnarray,hyperref,amsfonts,amsmath,amsthm,setspace,tikz,verbatim,charter}
\usepackage[a4paper,margin=1.9cm,top=2.5cm,bottom=2.5cm,centering,vcentering]{geometry}
\definecolor{ao(english)}{rgb}{0.13, 0.55, 0.13}

\hypersetup{
  colorlinks=true,
  linkcolor=ao(english),
  citecolor=ao(english),
  urlcolor=ao(english)
}

\newtheorem{theorem}{Theorem}
\newtheorem{conjecture}{Conjecture}
\newtheorem{corollary}{Corollary}
\newtheorem{lemma}{Lemma}

\numberwithin{equation}{section}

\title[Arithmetic properties of partition functions introduced by Pushpa and Vasuki]{Arithmetic properties of partition functions introduced by Pushpa and Vasuki}

\author[H. Nath]{Hemjyoti Nath}
\address[H. Nath]{Department of Mathematics, University of Florida, P.O. Box 118105, Gainesville, FL 32611-8105, USA}
\email{h.nath@ufl.edu}

\author[M. P. Saikia]{Manjil P. Saikia}
\address[M. P. Saikia]{Mathematical and Physical Sciences division, School of Arts and Sciences, Ahmedabad University, Ahmedabad 380009, Gujarat, India}
\email{manjil@saikia.in}



\begin{document}

\begin{abstract}
    In this short note, we prove several infinite family of congruences for some restricted partitions introduced by Pushpa and Vasuki (2022) (thereby, also proving a conjecture of Dasappa et. al. (2023)). We also prove some isolated congruences which seem to have been missed by earlier authors. Our proof techniques uses both elementary means as well as the theory of modular forms.
\end{abstract}

\maketitle

\section{Introduction}

In a recent paper, Pushpa and Vasuki \cite{PushpaVasuki} proved Eisenstein series identities of level $5$ of weight $2$ due to Ramanujan and some new identities for level $7$. In the course of their investigations, they introduced seven restricted color partition functions, which are the objects of study in this short note. A partition of an integer $n$ is a non-increasing sequence $\lambda =(\lambda_1, \lambda_2, \ldots, \lambda_k)$ such that $\sum\limits_{i=1}^k\lambda_i=n$. For instance the $5$ partitions of $4$ are 
\[
(4), (3,1), (2,2), (2,1,1), \text{ and } (1,1,1,1).
\]

Partitions have been studied since at least the time of Euler, who gave their generating function
\[
\sum_{n\geq 0}p(n)q^n=\frac{1}{\prod_{i\geq 1}(1-q^i)}=\frac{1}{(q;q)_\infty}=\frac{1}{f_1},
\]
where $p(n)$ is the number of partitions of $n$ and we use the notation
\[
(a;q)_\infty=\prod_{i\geq 0}(1-aq^i) \quad \text{and} \quad f_k:=(q^k;q^k)_\infty.
\] The generating functions of the seven classes of partitions introduced by Pushpa and Vasuki \cite{PushpaVasuki} are as follows
\begin{align}
    \sum_{n\geq 0}P^\ast(n)q^n&=f_1^4f_5^4,\label{gf-p}\\
    \sum_{n\geq 0}M(n)q^n&=\frac{f_2^5f_5^5}{f_1f_{10}},\label{gf-m}\\
    \sum_{n\geq 0}T^\ast(n)q^n&=\frac{f_1^5f_{10}^5}{f_2f_5},\label{gf-t}\\
    \sum_{n\geq 0}A(n)q^n&=\frac{f_2^6f_7^6}{f_1^2},\label{gf-a}\\
    \sum_{n\geq 0}B(n)q^n&=\frac{f_1^6f_{14}^4}{f_2^2f_7^2},\label{gf-b}\\
    \sum_{n\geq 0}K(n)q^n&=f_1^2f_2^2f_7^2f_{14}^2,\label{gf-k}\\
    \sum_{n\geq 0}L(n)q^n&=\frac{f_1^5f_7^5}{f_2f_{14}}.\label{gf-l}
\end{align}

Pushpa and Vasuki \cite{PushpaVasuki} proved some isolated congruences satisfied by these functions. In a follow-up work
Dasappa \textit{et. al.} \cite{DasappaChannabasavayyaKeerthana} proived several more congruences for these classes of functions and they also gave some infinite families of congruences. In addition, they gave the following conjecture.
\begin{conjecture}\cite[Conjecture 7.1]{DasappaChannabasavayyaKeerthana}\label{conjd}
    For all $n\geq 0$ and $\alpha\geq 1$, we have
    \[
    K(7^\alpha n+7^\alpha -2)\equiv 0 \pmod{7^\alpha}.
    \]
\end{conjecture}

\noindent Motivated by the work of Dasappa \textit{et. al.} \cite{DasappaChannabasavayyaKeerthana} and the above conjecture, we prove several infinite family of congruences for the partitions introduced by Pushpa and Vasuki \cite{PushpaVasuki}. We also prove some isolated congruences which seem to have been missed by earlier authors. Our proof techniques uses both elementary means as well as the theory of modular forms. Using a mixture of algorithmic and elementary techniques we give a proof of Conjecture \ref{conjd} as well.

Our first result is the following.
\begin{theorem}\label{thm 16n+7}
    For all $n\geq 0$, we have
    \begin{align}
        P^\ast(2n+1)&\equiv 0 \pmod 4,\label{p2n+1}\\
        P^\ast(4n+3)&\equiv 0 \pmod 8, \label{p4n+3}\\
        P^\ast(16n+7)&=0.\label{p16n+7}
    \end{align}
\end{theorem}

\noindent Theorem \ref{thm 16n+7} is proved in Section \ref{secthm16}. We note that Pushpa and Vasuki \cite[Eq. (5.32)]{PushpaVasuki} had also proved \eqref{p2n+1}, but the other two results do not appear in their work.

\begin{theorem}\label{P}
For all $n\geq 0$, we have
    \[
    P^\ast(16n+15) = -64 \cdot P^\ast(n).
    \]
\end{theorem}

\noindent Theorem \ref{P} is proved in Section \ref{secP}. Here, we prove the following consequence of Theorem \ref{P}.

\begin{corollary}\label{P0}
    For all $n\geq 0$ and $\alpha\geq 1$, we have
    \[
    P^\ast(2^\alpha n+2^\alpha -1)\equiv 0 \pmod{2^{\alpha+1}}.
    \]
\end{corollary}

\begin{proof}
    The cases $\alpha=1, 2$ are proved in Theorem \ref{thm 16n+7} above, while the case $\alpha=3$ is proved in Corollary \ref{cor 8n+7} later. The result now follows via induction on $\alpha$ and an application of Theorem \ref{P} once we notice that, for all $\alpha\geq 3$ we have
    \[
2^{\alpha+1}n+2^{\alpha+1}-1=16(2^{\alpha-3}n+2^{\alpha-3}-1)+15.
    \]
\end{proof}

\begin{theorem}\label{Theorem M}
For all $n\geq 0$ and $\alpha\geq 1$, we have
    \begin{equation}
        \sum_{n\geq 0}M(5^{\alpha}n+5^{\alpha}-1)q^n=5^{\alpha}\Psi_{\alpha},\label{gf-m1}
    \end{equation}
where
\begin{equation*}
    \Psi_{\alpha} = 
    \begin{cases}
        q\dfrac{f_1^5 f_{10}^5}{f_2 f_5}, & \text{if } \alpha \text{ is odd}, \\[10pt]
        \dfrac{f_2^5 f_5^5}{f_1 f_{10}}, & \text{if } \alpha \text{ is even}.
    \end{cases}
\end{equation*}
\end{theorem}

\noindent An immediate consequence of the above result is the following corollary.
\begin{corollary}\label{cor:m0}
For all $n\geq 0$ and $\alpha\geq 1$, we have
    \[
    M(5^\alpha n+5^\alpha -1)\equiv 0 \pmod{5^{\alpha}}.
    \]
\end{corollary}

\noindent Corollary \ref{cor:m0} generalizes a result of Pushpa and Vasuki \cite[Theorem 5.13]{PushpaVasuki}, where they prove the $\alpha=1$ case.

\begin{theorem}\label{Theorem T}
For all $n\geq 0$, we have
    \begin{equation}
        \sum_{n\geq 0}T^{\ast}(5n+3)q^n=5\sum_{n\geq 0}M(n)q^n.\label{gf-m2}
    \end{equation}
\end{theorem}

\begin{corollary}\label{cor t0}
For all $n\geq 0$ and $\alpha\geq 1$, we have
    \[
    T^\ast(5^\alpha n+5^\alpha -2)\equiv 0 \pmod{5^{\alpha}}.
    \]
\end{corollary}

\begin{proof}
This result follows easily from Theorem \ref{Theorem T} and Corollary \ref{cor:m0}, by observing the following
\[
T^\ast(5^\alpha n+5^\alpha -2)=T^\ast(5(5^{\alpha-1}n+5^{\alpha-1}-1)+3)=5\cdot M(5^{\alpha-1}n+5^{\alpha-1}-1) \equiv 0 \pmod {5\cdot 5^{\alpha-1}}.
\]
\end{proof}

\noindent Corollary \ref{cor t0} generalizes a result of Pushpa and Vasuki \cite[Theorem 5.14]{PushpaVasuki}, where they prove the $\alpha=1$ case. Theorems \ref{Theorem M} and \ref{Theorem T} are proved in Section \ref{secMT}.

\begin{theorem}\label{Theorem 9n, 10n}
For all $n\geq 0$, we have
    \begin{align}
    P^\ast\left( 2 \cdot 5^{\alpha}n + 2 \cdot 5^{\alpha} -1\right)&\equiv 0\pmod{2^2 \cdot 5^{\alpha}}.\label{P(10n+9)}
    \end{align}
\end{theorem}

\noindent Theorem \ref{Theorem 9n, 10n} is proved in Section \ref{secthm9n}.

\begin{theorem}\label{abradu}
    For all $n\geq 0$, we have
    \begin{align}
        A(7n+6)&\equiv 0 \pmod{7},\label{aradu}\\
        B(7n+4)&\equiv 0 \pmod{7}.\label{bradu}
    \end{align}
\end{theorem}

\noindent Theorem \ref{abradu} is proved in Section \ref{proof:abradu}.

Our final theorem gives a proof of Conjecture \ref{conjd} as well as another infinite family of congruences.
\begin{theorem}\label{thmconjd}
            For all $n\geq 0$ and $\alpha\geq 1$, we have
\begin{equation}\label{eqconjd}
    K(7^\alpha n+7^\alpha -2)\equiv 0 \pmod{7^\alpha},
\end{equation}
and
\begin{equation}\label{eqconjdd}
        L(7^\alpha n+7^\alpha -1)\equiv 0 \pmod{7^\alpha}.
\end{equation}
\end{theorem}
\noindent Theorem \ref{thmconjd} is proved in Section \ref{sec:conjd}.

The rest of the paper is organized as follows: we begin with some preliminaries in Section \ref{sec:prelim}, then prove our results in Sections \ref{secthm16} -- \ref{sec:conjd}, and finally close the paper with some remarks in Section \ref{sec:conc}.

\section{Preliminaries}\label{sec:prelim}

\subsection{Elementary $q$-series identities}

Ramanujan's general theta function $f(a,b)$ is defined by 
\begin{align*}
	f(a,b) = \sum_{k=-\infty}^{\infty} a^{\frac{k(k+1)}{2}}b^{\frac{k(k-1)}{2}}, \quad |ab| < 1.
\end{align*}
Two special cases of $f(a,b)$ are
\begin{align}
	\varphi(q) &:= f(q,q) = \sum_{k=-\infty}^\infty q^{k^2}=1+2\sum_{n\geq 1}q^{n^2},\label{varphi}
\end{align}
and
\begin{align*}
    \psi(q):=f(q,q^3)=\sum_{k=0}^\infty q^{k(k+1)/2}.
\end{align*}

We will need the following identities in order to prove our theorems.

\begin{lemma}\label{lem1}
Let $$R(q)=\dfrac{(q;q^5)_{\infty}(q^4;q^5)_{\infty}}{(q^2;q^5)_{\infty}(q^3;q^5)_{\infty}} \quad \text{and} \quad R_j=R(q^j),$$ then
\begin{enumerate}[(i)]
    \item \( f_1 = f_{25}\left( \dfrac{1}{R(q^5)} -q -q^2R(q^5)\right), \label{eq:f1_def}\)
    \item[]\vspace{10pt}

    \item \( \psi(q)=\dfrac{f_2^2}{f_1}=f(q^{10},q^{15})+qf(q^5,q^{20})+q^3\psi(q^{25}), \label{eq:psi_q}\)
    \item[]\vspace{10pt}

    \item \( \varphi(-q) = \dfrac{f_1^2}{f_2}  = -2qf(-q^{15},-q^{35})+2q^4f(-q^{5},-q^{45})+ \varphi(-q^{25}). \label{eq:phi_neg_q}\)
\end{enumerate}
\end{lemma}

\begin{proof}
    See \cite[Equation 8.1.1]{Hirschhorn17} for equation~\eqref{eq:f1_def},
\cite[p.~262, Entry 10~(i)]{BerndtIII} for equation~\eqref{eq:psi_q}, and 
\cite[p.~262, Entry 10~(ii)]{BerndtIII} for equation~\eqref{eq:phi_neg_q}.
\end{proof}

We also make use of the following identities.

\begin{lemma}
\begin{align}
    f_1^4&=\frac{f_4^{10}}{f_2^2f_8^4}-4q\frac{f_2^2f_8^4}{f_4^2},\label{2-d f_1^4}\\
    \frac{f_5}{f_1}&=\frac{f_8f_{20}^2}{f_2^2f_{40}}+q\frac{f_4^3f_{10}f_{40}}{f_2^3f_8f_{20}}, \label{2-d f_5/f_1}\\
    \frac{f_1}{f_5}&=\frac{f_2f_8f_{20}^3}{f_4f_{10}^3f_{40}}-q\frac{f_4^2f_{40}}{f_8f_{10}^2}, \label{2-d f_1/f_5}\\
    f_1f_5^3&=f_2^3f_{10}-q\frac{f_2^2f_{10}^2f_{20}}{f_4}+2q^2f_4f_{20}^3-2q^3\frac{f_4^4f_{10}f_{40}^2}{f_2f_8^2}, \label{2-d f_1f_5^3}\\
    f_1^3f_5 & =\frac{f_2^2f_4f_{10}^2}{f_{20}}+q\left(2f_4^3f_{20}-5f_2f_{10}^3\right)+2q^2 \frac{f_4^6f_{10}f_{40}^2}{f_2f_8^2f_{20}^2}. \label{2-d f_1^3f_5}
\end{align}
\end{lemma}

\begin{proof}
    Equation \eqref{2-d f_1^4} is \cite[Eq. (2.4)]{BaruahKaur}. Equation \eqref{2-d f_5/f_1} is \cite[Theorem 2.1]{HirSel}. Equation \eqref{2-d f_1f_5^3} is \cite[p. 315]{BerndtIII}. Equation \eqref{2-d f_1^3f_5} can be found in \cite{Naika}.

Replacing $q$ by $-q$ in \eqref{2-d f_5/f_1} and using the fact that $(-q; -q)_\infty = \dfrac{f_2^3}{f_1 f_4}$, we obtain \eqref{2-d f_1/f_5}.

\end{proof}

\subsection{Radu's algorithm}\label{sec:rk}
To prove some of our results, we use Smoot's \cite{Smoot} implementation of an algorithm of Radu \cite{Radu} which we describe now. Radu's algorithm can be used to prove Ramanujan type congruences of the form stated in the previous section. The algorithm takes as an input the generating function
\[
\sum_{n\geq 0}a_r(n)q^n=\prod_{\delta|M}\prod_{n\geq 1}(1-q^{\delta n})^{r_\delta},
\]
and positive integers $m$ and $N$, with $M$ another positive integer and $(r_\delta)_{\delta|M}$ is a sequence indexed by the positive divisors $\delta$ of $M$. With this input, Radu's algorithm tries to produce a set $P_{m,j}(j)\subseteq \{0,1,\ldots, m-1\}$ which contains $j$ and is uniquely defined by $m, (r_\delta)_{\delta|M}$ and $j$. Then, it decides if there exists a sequence $(s_\delta)_{\delta |N}$ such that
\[
q^\alpha \prod_{\delta|M}\prod_{n\geq 1}(1-q^{\delta n})^{s_\delta} \cdot \prod_{j^\prime \in P_{m,j}(j)}\sum_{n\geq 0}a(mn+j^\prime)q^n,
\]
is a modular function with certain restrictions on its behaviour on the boundary of $\mathbb{H}$.

Smoot \cite{Smoot} implemented this algorithm in Mathematica and we use his \texttt{RaduRK} package which requires the software packaage \texttt{4ti2}. Documentation on how to intall and use these packages are available from Smoot \cite{Smoot}. We use this implemented \texttt{RaduRK} algorithm to prove Theorem \ref{abradu}.

It is natural to guess that $N=m$ (which corresponds to the congruence subgroup $\Gamma_0(N)$), but this is not always the case, although they are usually closely related to one another. The determination of the correct value of $N$ is an important problem for the usage of \texttt{RaduRK} and it depends on a criterion called the $\Delta^\ast$ criterion \cite[Definitions 34 and 35]{Radu}, which we do not explain here. It is easy to check the minimum $N$ which satisfies this criterion by running \texttt{minN[M, r, m, j]}. %which we do now for our generating function. The generating function of $\eou(n)$ given in \eqref{gf:eou} can be described by setting $M=4$ and $r=\{0,-2,2\}$. We are interested in only the case $m=14$ and $j=8$ for proving the generating function given in \eqref{Gen 14n+8}. Running \texttt{minN[4, \{0,-2,2\}, 14, 8]} yields $56$ as the output for the minimum choices of $N$. This values of $N$ is large enough for the computation to take quite a while in a modest laptop. So, we are going to use the generating function for $\eou(2n)$ in Section \ref{sec:pf} to prove \eqref{Gen 14n+8} and Theorem \ref{thm:cong64}.

\section{Proof of Theorem \ref{thm 16n+7}}\label{secthm16}

From \eqref{gf-p}, we have
\begin{equation*}
    \sum_{n \geq 0} P^\ast(n) q^n = f_1^4 f_5^4.
\end{equation*}
Using \eqref{2-d f_1^4} and then replacing $q$ with $q^5$ in that equation, and together using both equations in the above expression, we obtain
\[
 \sum_{n \geq 0} P^\ast(n) q^n = \left(\frac{f_4^{10}}{f_2^2f_8^4}-4q\frac{f_2^2f_8^4}{f_4^2}\right)\left(\frac{f_{20}^{10}}{f_{10}^2f_{40}^4}-4q\frac{f_{10}^2f_{40}^4}{f_{20}^2}\right).
\]
In the above, extracting the terms involving the odd powers of $q$, we obtain
\begin{equation*}
    \sum_{n \geq 0} P^\ast(2n+1) q^n = -4 \left( \dfrac{f_1^2f_4^4f_{10}^{10}}{f_2^2f_5^2f_{20}^4} + q^2\dfrac{f_2^{10}f_5^2f_{20}^4}{f_1^2f_4^4f_{10}^2} \right).
\end{equation*}
This proves \eqref{p2n+1}.

We rewrite the above equation as
\[
    \sum_{n \geq 0} P^\ast(2n+1) q^n = -4 \left( \dfrac{f_4^4f_{10}^{10}}{f_2^2f_{20}^4}\cdot \left(\frac{f_1^2}{f_5^2}\right)^2 + q^2\dfrac{f_2^{10}f_{20}^4}{f_4^4f_{10}^2}\cdot \left(\frac{f_5^2}{f_1^2}\right)^2 \right).
\]
Using \eqref{2-d f_1/f_5} and \eqref{2-d f_5/f_1} together in the above expression, and then extracting the terms involving the odd powers of $q$, we obtain
\begin{equation*}
    \sum_{n \geq 0} P^\ast(4n+3) q^n = 8 \left( \dfrac{f_2^5f_5^5}{f_1f_{10}} -q \dfrac{f_1^5f_{10}^5}{f_2f_5} \right) .
\end{equation*}
This proves \eqref{p4n+3}.

We rewrite the above equation as
\[
    \sum_{n \geq 0} P^\ast(4n+3) q^n = 8 \left( \dfrac{f_2^5}{f_{10}}\cdot f_5^4\cdot \frac{f_5}{f_1} -q \dfrac{f_{10}^5}{f_2}\cdot f_1^4\cdot \frac{f_1}{f_5} \right) .
\]
Using the magnified version $q \to q^5$ in \eqref{2-d f_1^4}, and also using \eqref{2-d f_1/f_5} and \eqref{2-d f_5/f_1} all together in the above expression, then extracting the terms involving the odd powers of $q$, we obtain
\begin{equation*}
    \sum_{n \geq 0} P^\ast(8n+7) q^n = 8 \left( \dfrac{f_1^2f_2^3f_{10}^9}{f_4f_5^2f_{20}^3} - \dfrac{f_2^9f_5^2f_{10}^3}{f_1^2f_4^3f_{20}}-4qf_1f_4^3f_5^3f_{20} -4q^2f_1^3f_4f_5f_{20}^3 \right) .
\end{equation*}
We rewrite this as
\[
 \sum_{n \geq 0} P^\ast(8n+7) q^n = 8 \left( \dfrac{f_2^3f_{10}^9}{f_4f_{20}^3}\cdot \left(\frac{f_1}{f_5}\right)^2 - \dfrac{f_2^9f_{10}^3}{f_4^3f_{20}}\cdot \left(\frac{f_5}{f_1}\right)^2-4qf_4^3f_{20}\cdot f_1f_5^3 -4q^2f_4f_{20}^3 \cdot f_1^3f_5\right) .
\]
Using \eqref{2-d f_1/f_5}, \eqref{2-d f_5/f_1}, \eqref{2-d f_1f_5^3}, and \eqref{2-d f_1^3f_5} all together in the above expression, then extracting the terms involving the even powers of $q$, we obtain
\begin{equation*}
    \sum_{n \geq 0} P^\ast(16n+7) q^n = 0 .
\end{equation*}
This proves \eqref{p16n+7}.

\section{Proof of Theorem \ref{P}}\label{secP}

We first start with the following lemma. Here, we define the ``huffing" operator $H_k$ for a power series $\sum_{n\geq 0}P(n)q^n$ and a positive integer $k$ as follows
\[
H_k\left(\sum_{n\geq 0}P(n)q^n\right):=\sum_{n\geq 0}P(kn)q^{kn}.
\]

\begin{lemma}\label{Lemma 1}
If we define
    \begin{equation*}
        \lambda(q):=qf_1^4f_5^4, \quad \theta:=\frac{\lambda^3(q^2)}{\lambda(q)\lambda(q^4)(\lambda(q^2)+2\lambda(q^4))}, \quad\delta := \frac{\lambda^3(q^2)}{\lambda(q^4)(\lambda(q^2)+2\lambda(q^4))^2},
    \end{equation*}
then we obtain
\begin{equation}\label{el1}
    H_2( \lambda(q) ) = -4\lambda(q^2)-8\lambda(q^4).
\end{equation}

\end{lemma}

\begin{proof}
We recall the following modular equation of degree 5, originally due to Jacobi \cite{jacobi1,jacobi2}. Using the notation in Berndt \cite{Bruce}, if $\beta$ is of degree 5 over $\alpha$, then the identity  
\begin{equation*}
    (\alpha \beta)^{1/2} + \left\{(1 - \alpha)(1 - \beta)\right\}^{1/2} + 2 \left\{16\alpha \beta (1 - \alpha)(1 - \beta)\right\}^{1/6} = 1
\end{equation*}
holds. Furthermore, Berndt \cite[Equation 4.4]{Bruce} demonstrated that this modular equation can be reformulated as the following $q$-product identity
\begin{equation*}
    (-q; q^2)^4_{\infty} (-q^5; q^{10})^4_{\infty} - (q; q^2)^4_{\infty} (q^5; q^{10})^4_{\infty} = 8q + 16q^3 (-q^2; q^2)^4_{\infty} (-q^{10}; q^{10})^4_{\infty},
\end{equation*}
which can be rewritten as
\begin{equation}\label{Ber modular 1}
    \frac{f_2^8f_{10}^8}{f_1^4f_5^4f_4^4f_{20}^4}-\frac{f_1^4f_5^4}{f_2^4f_{10}^4}=8q+16q^3\frac{f_4^4f_{20}^4}{f_2^4f_{10}^4}.
\end{equation}
One may also express the above identity as
    \begin{equation*}
        \frac{f_1^4f_5^4}{q(f_2^4f_{10}^4+2q^2f_4^4f_{20}^4)} = \frac{f_2^{12}f_{10}^{12}}{qf_1^4f_5^4f_4^4f_{20}^4(f_2^4f_{10}^4+2q^2f_4^4f_{20}^4)}-8.
    \end{equation*}
In our notation, the above identity is equivalent to
\begin{equation}\label{theta^i 1}
    \frac{\delta}{\theta} = \theta-8.
\end{equation}

Let 
\begin{equation*}
    R(q):= \frac{f_2^8f_{10}^8}{f_1^4f_5^4f_4^4f_{20}^4}, \quad S(q):= \frac{f_1^4f_5^4}{f_2^4f_{10}^4},
\end{equation*}
it then follows readily that
\begin{equation*}
    R(q)=S(-q), \quad S(q)=R(-q).
\end{equation*}
Therefore, using \eqref{Ber modular 1}, we obtain
\begin{equation*}
    \frac{R(q)-S(q)}{2}=4q+8q^3\frac{f_4^4f_{20}^4}{f_2^4f_{10}^4}.
\end{equation*}
Moreover, using the fact that
\begin{equation*}
    R(q) = \frac{R(q)+S(q)}{2}+\frac{R(q)-S(q)}{2},
\end{equation*}
we obtain
\begin{equation*}
    \frac{f_2^8f_{10}^8}{f_1^4f_5^4f_4^4f_{20}^4} = \varepsilon(q^2) + 4q+8q^3\frac{f_4^4f_{20}^4}{f_2^4f_{10}^4}.
\end{equation*}
where $\varepsilon(q^2)$ is a function of $q^2$. The above equation can be rewritten as
\begin{equation}\label{e1}
    \frac{f_2^{12}f_{10}^{12}}{qf_1^4f_5^4f_4^4f_{20}^4(f_2^4f_{10}^4+2q^2f_4^4f_{20}^4)} = \frac{\varepsilon(q^2)f_2^4f_{10}^4}{q(f_2^4f_{10}^4+2q^2f_4^4f_{20}^4)} + 4.
\end{equation}
It is evident from \eqref{e1} that
\begin{equation}\label{e2}
    H_2(\theta) = 4.
\end{equation}
Using \eqref{e2} in \eqref{theta^i 1}, we obtain
\begin{equation*}
    H_2\left(\frac{1}{\theta}\right) = -\dfrac{4}{\delta}.
\end{equation*}
By substituting the expressions for $\theta$ and $\delta$ into the above equation and simplifying, we complete the proof of Lemma~\ref{Lemma 1}.

\end{proof}

We now prove Theorem \ref{P}. We use the notations of the previous lemma without commentary.
\begin{proof}[Proof of Theorem \ref{P}]
We have
\begin{equation}\label{e18}
     \sum_{n \geq 0} P^{\ast}(n) q^{n+1} = \lambda(q).
\end{equation}
Extracting the even powers of \( q \) from \eqref{e18}, we obtain
\begin{equation*}
    \sum_{n \geq 1} P^\ast(2n - 1) q^{2n} = H_2\!\left( \lambda(q) \right).
\end{equation*}
Using \eqref{el1}, we get
\begin{equation*}
    \sum_{n \geq 1} P^\ast(2n - 1) q^{2n} = -4\lambda(q^2) - 8\lambda(q^4).
\end{equation*}

Replacing \( q^2 \) by \( q \) and then \( n \) by \( n + 1 \) in the above identity, we obtain
\begin{equation}\label{e19}
    \sum_{n \geq 0} P^\ast(2n + 1) q^{n+1} = -4\lambda(q) - 8\lambda(q^2).
\end{equation}
Again, extracting the even powers of \( q \) from \eqref{e19}, we obtain
\begin{equation*}
    \sum_{n \geq 1} P^\ast(4n - 1) q^{2n} = -4 \cdot H_2\!\left( \lambda(q) \right) - 8\lambda(q^2).
\end{equation*}
Using \eqref{el1}, we get
\begin{equation*}
    \sum_{n \geq 1} P^\ast(4n - 1) q^{2n} = 8\lambda(q^2) + 32\lambda(q^4).
\end{equation*}

Replacing \( q^2 \) by \( q \) and then \( n \) by \( n + 1 \) in the above identity, we obtain
\begin{equation}\label{e20}
    \sum_{n \geq 0} P^\ast(4n + 3) q^{n+1} = 8\lambda(q) + 32\lambda(q^2).
\end{equation}
Extracting the even powers of $q$ from \eqref{e20}, we obtain
\[
\sum_{n \geq 1} P^\ast(8n -1) q^{2n} = 8\cdot H_2(\lambda(q)) + 32\lambda(q^2).
\]
Again, like before using \eqref{el1}, replacing $q^2$ by $q$ and then $n$ by $n+1$ in the above, we obtain
\begin{equation}\label{p8n+7}
    \sum_{n \geq 0} P^\ast(8n + 7) q^{n+1}  = -64\lambda(q^2).
\end{equation}
Once again, extracting the even powers of $q$ we arrive readily at
\begin{equation}\label{p16n+15}
    \sum_{n \geq 0} P^\ast(16n + 15) q^{n+1}  = -64\lambda(q).
\end{equation}
Equations \eqref{e18} and \eqref{p16n+15} proves the result.
\end{proof}

\begin{corollary}\label{cor 8n+7}
    For all $n\geq 0$, we have
    \begin{align*}
    P^\ast(8n + 7) &\equiv 0 \pmod{64},\\
    P^\ast(32n + 31) &\equiv 0 \pmod{256},\\
    P^\ast(64n + 63) &\equiv 0 \pmod{512}.
\end{align*}
\end{corollary}

\begin{proof}
    Equation \eqref{p8n+7} proves the first congruence. The last two can be proved in a similar fashion, where we will obtain the following generating functions:
\begin{align*}
    \sum_{n \geq 0} P^\ast(32n + 31) q^{n+1} &= 256\lambda(q) + 512\lambda(q^2),\\
    \sum_{n \geq 0} P^\ast(64n + 63) q^{n+1} &= -512\lambda(q) - 2048\lambda(q^2).
\end{align*}
\end{proof}

\section{Proof of Theorems \ref{Theorem M} and \ref{Theorem T}}\label{secMT}

\begin{proof}[Proof of Theorem \ref{Theorem M}]
    We prove the theorem using induction on $\alpha$.

From \eqref{gf-m}, we have
\begin{equation*}
    \sum_{n \geq 0} M(n)q^n = \dfrac{f_2^5 f_5^5}{f_1 f_{10}} = \dfrac{f_2^2}{f_1} \cdot f_2^3 \cdot \dfrac{f_5^5}{f_{10}}.
\end{equation*}
Now, using Lemma \ref{lem1} in the equation above, we obtain
\[
\sum_{n \geq 0} M(n)q^n = \left(f(q^{10},q^{15})+qf(q^5,q^{20})+q^3\psi(q^{25})\right)f_{50}^3\left(\frac{1}{R(q^{10})}-q^2-q^4R(q^{10})\right)^3\frac{f_5^5}{f_{10}}.
\]
We now extract terms of the form $q^{5n+4}$ to obtain
\begin{equation*}
    \sum_{n \geq 0} M(5n+4)q^n = 5q \dfrac{f_1^5 f_{10}^5}{f_2 f_5},
\end{equation*}
which corresponds to the case $\alpha = 1$.

Next, applying Lemma \ref{lem1} again into the equation above, we again extract terms of the form $q^{5n+4}$ to obtain
\begin{equation*}
    \sum_{n \geq 0} M(25n+24)q^n = 25  \dfrac{f_2^5 f_5^5}{f_1 f_{10}},
\end{equation*}
which corresponds to the case $\alpha = 2$.

Now, we assume the result holds for $\alpha = k$, that is
\begin{equation*}
    \sum_{n \geq 0} M(5^{k}n + 5^{k} - 1)q^n = 5^{k} \Psi_{k},
\end{equation*}
where
\begin{equation*}
    \Psi_{k} =
    \begin{cases}
        q \dfrac{f_1^5 f_{10}^5}{f_2 f_5}, & \text{if } k \text{ is odd}, \\[10pt]
        \dfrac{f_2^5 f_5^5}{f_1 f_{10}}, & \text{if } k \text{ is even}.
    \end{cases}
\end{equation*}
If $k$ is odd, we proceed as in the case $\alpha = 2$, and if $k$ is even, we proceed as in the case $\alpha = 1$. Therefore, the result also holds for $\alpha = k+1$.

This completes the proof of Theorem~\ref{Theorem M}.
\end{proof}

\begin{proof}[Proof of Theorem \ref{Theorem T}]
   From \eqref{gf-t}, we have
\begin{equation*}
    \sum_{n \geq 0} T^\ast(n) q^n = \dfrac{f_1^5 f_{10}^5}{f_2 f_5} = \dfrac{f_1^2}{f_2}\cdot f_1^3 \cdot \frac{f_{10}^5}{f_5}.
\end{equation*}
Now, using Lemma \ref{lem1} in the above equation like we did in the previous proof, and extracting the terms of the form $q^{5n+3}$, we obtain
\begin{equation*}
    \sum_{n \geq 0} T^\ast(5n+3) q^n = 5\frac{f_2^5 f_5^5}{f_1 f_{10}} = 5\sum_{n \geq 0} M(n) q^n.
\end{equation*}
This completes the proof of Theorem~\ref{Theorem T}.
\end{proof}

\section{Proof of Theorem \ref{Theorem 9n, 10n}}\label{secthm9n}

We need the following result.

\begin{lemma}
    For all $n\geq 0$ and $\alpha\geq 1$, we have
\begin{equation}\label{e17}
  \sum_{n\geq 0}  P^\ast(5^\alpha n+5^\alpha -1)q^n = (-5)^{\alpha}f_1^4f_5^4.
\end{equation}
\end{lemma}

\begin{proof}
We prove the result by induction on $\alpha$. From \eqref{gf-p}, we have
\begin{equation*}
    \sum_{n \geq 0} P^\ast(n) q^n = f_1^4 f_5^4.
\end{equation*}
Now, using Lemma \ref{lem1} \eqref{eq:f1_def} and extracting the terms of the form $q^{5n+4}$, we obtain
\begin{equation*}
    \sum_{n \geq 0} P^\ast(5n+4) q^n = -5 f_1^4 f_5^4,
\end{equation*}
which corresponds to the case $\alpha = 1$.

Next, we assume that the theorem holds for $\alpha = k$, that is, we have
\begin{equation*}
     \sum_{n\geq 0}  P^\ast(5^k n+5^k -1) q^n= (-5)^{k}f_1^4f_5^4.
\end{equation*}
Now, applying Lemma \ref{lem1} \eqref{eq:f1_def} again and extracting the terms of the form $q^{5n+4}$ from the above equation, we obtain
\begin{equation*}
     \sum_{n\geq 0}  P^\ast(5^{k+1} n+5^{k+1} -1) q^n= (-5)^{k+1}f_1^4f_5^4.
\end{equation*}
which corresponds to the case $\alpha = k+1$. Hence, the lemma holds for all $\alpha \geq 1$.
\end{proof}

\begin{proof}[Proof of Theorem \ref{Theorem 9n, 10n}]
Applying the magnification $q \mapsto q^5$ in \eqref{2-d f_1^4}, and substituting it together with \eqref{2-d f_1^4} into \eqref{e17}, we then extract terms of the form $q^{2n+1}$. This yields
\begin{equation*}
    \sum_{n \geq 0} P^\ast\!\left( 2 \cdot 5^{\alpha}n + 2 \cdot 5^{\alpha} - 1 \right) q^n 
    = (-1)^{\alpha+1} \cdot 2^2 \cdot 5^{\alpha} \left( 
        \dfrac{f_1^2 f_4^4 f_{10}^{10}}{f_2^2 f_5^2 f_{20}^4} 
        + q^2 \dfrac{f_2^{10} f_5^2 f_{20}^4}{f_1^2 f_4^4 f_{10}^2} 
    \right).
\end{equation*}
This completes the proof of \eqref{P(10n+9)}.
\end{proof}

\section{Proof of Theorem \ref{abradu}}\label{proof:abradu}

Since the proofs of \eqref{aradu} and \eqref{bradu} are the same, we only mention the details of the proof of \eqref{aradu}. The Mathematica output for both the cases can be found in this URL: \url{https://manjilsaikia.in/publ/mathematica/pv.nb}.

\begin{proof}[Proof of \eqref{aradu}]
Using \eqref{gf-a}, we first run \texttt{minN[14,\{-2,6,6,0\},7,6]} which gives us $N=14$, which is handled easily in a regular laptop. Radu's algorithm now gives a proof of \eqref{aradu}. Here we give the output of \texttt{RK}.
\[
\begin{array}{c|c}
 \text{N:} & 14 \\
\hline
 \text{$\{$M,(}r_{\delta })_{\delta |M}\text{$\}$:} & \{14,\{-2,6,6,0\}\} \\
\hline
 \text{m:} & 7 \\
\hline
 P_{m,r}\text{(j):} & \{6\} \\
\hline
 f_1\text{(q):} & \dfrac{f_2f_7^9}{q^4f_1^7f_{14}^{13}}\\
\hline
 \text{t:} & \dfrac{f_2f_7^7}{q^2f_1f_{14}^7} \\
\hline
 \text{AB:} & \left\{1,\dfrac{f_2^8f_7^4}{q^3f_1^4f_{14}^8}-\dfrac{4f_2f_7^7}{q^2f_1f_{14}^7} \right\} \\
\hline
 \left\{p_g\text{(t): g$\in $AB$\}$}\right. & \{-7 t,0\} \\
\hline
 \text{Common Factor:} & 7 \\
\end{array}
\]

The interpretation of this output is as follows.

The first entry in the procedure call \texttt{RK[14, 14, \{-2, 6, 6, 0\}, 7, 6]} corresponds to specifying $N=14$, which fixes the space of modular functions
\[
M(\Gamma_0(N)):=\text{the algebra of modular functions for $\Gamma_0(N)$}.
\]

 The second and third entry of the procedure call \texttt{RK[14, 14, \{-2, 6, 6, 0\}, 7, 6]} gives the assignment $\{M, (r_\delta)_{\delta|M}\}=\{14, (-2, 6, 6,0)\}$, which corresponds to specifying $(r_\delta)_{\delta|M}=(r_1,r_2,r_7,r_{14})=(-2,6,6,0)$, so that
 \[
\sum_{n\geq 0}A(n)q^n=\prod_{\delta|M}(q^\delta;q^\delta)^{r_\delta}_\infty = \frac{f_2^6f_7^6}{f_1^2}.
 \]

 The last two entries of the procedure call \texttt{RK[14, 14, \{-2, 6, 6, 0\}, 7, 6]} corresponds to the assignment $m=7$ and $j=6$, which means that we want the generating function
 \[
\sum_{n\geq 0}A(mn+j)q^n=\sum_{n\geq 0}A(8n+7)q^n.
 \]
 So, $P_{m,r}(j)=P_{7,r}(6)$ with $r=(-2,6,6,0)$.

The output $P_{m,r}(j):=P_{7,(-2,6,6,0)}(6)=\{6\}$ means that there exists an infinite product
\[
f_1(q)=\dfrac{f_2f_7^9}{q^4f_1^7f_{14}^{13}},
\]
such that
\[
f_1(q)\sum_{n\geq 0}A(8n+7)q^n\in M(\Gamma_0(14)).
\]

Finally, the output
\[
t=\dfrac{f_2f_7^7}{q^2f_1f_{14}^7}, \quad AB=\left\{1, \dfrac{f_2^8f_7^4}{q^3f_1^4f_{14}^8}-\dfrac{4f_2f_7^7}{q^2f_1f_{14}^7}\right\}, \quad \text{and}\quad \{p_g\text{(t): g$\in AB$\}},
\]
presents a solution to the question of finding a modular function $t\in M(\Gamma_0(14))$ and polynomials $p_g(t)$ such that
\[
f_1(q)\sum_{n\geq 0}A(7n+6)q^n =\sum_{g\in AB}p_g(t)\cdot g.
\]
In this specific case, we see that entries in the set $\{p_g\text{(t): g$\in AB$\}}$ has the common factor $7$, thus proving equation \eqref{aradu}.
\end{proof}

\noindent The interested reader can refer to \cite{Saikia}, \cite{AndrewsPaule2}, or \cite{SAIKIA_SARMA_2025} for more explanation of the method and how to read the output.

\section{Proof of Theorem \ref{thmconjd}}\label{sec:conjd}

Using the Mathematica implementation of Radu's algorithm, we find the following generating functions
\begin{align}
    \sum_{n \geq 0} L(7n+6)q^n & = -7\frac{f_1^5f_7^5}{f_2f_{14}} = -7\sum_{n \geq 0} L(n)q^n,\label{L}\\
    \sum_{n \geq 0} K(7n+5)q^n & = -7qf_1^2f_2^2f_7^2f_{14}^2 = -7q\sum_{n \geq 0} K(n)q^n\label{K}.
\end{align}
The relevant output is available in \url{https://manjilsaikia.in/publ/mathematica/kl.nb}, and the procedure to read the output is same as explained in the previous section, so we do not repeat it here.

By iterating the substitution $n \to 7n+6$ in \eqref{L}, we obtain the generating function for $\alpha \geq 1$ as
\begin{equation}\label{L1}
    \sum_{n \geq 0} L(7^{\alpha}n+7^{\alpha}-1)q^n = (-7)^{\alpha}\frac{f_1^5f_7^5}{f_2f_{14}}.
\end{equation}
From this, it follows immediately that
\begin{equation*}
    L(7^{\alpha}n+7^{\alpha}-1) \equiv 0 \pmod{7^{\alpha}},
\end{equation*}
thereby proving \eqref{eqconjd}. The proof of \eqref{L1} can be carried out by induction on $\alpha$, and is therefore omitted.

Similarly, by iterating the substitution $n \to 7n+5$ in \eqref{K}, we obtain the generating function for $\alpha \geq 1$ as
\begin{equation}\label{K1}
    \sum_{n \geq 0} K(7^{\alpha}n+7^{\alpha}-2)q^n = (-7)^{\alpha}q^{\alpha}f_1^2f_2^2f_7^2f_{14}^2.
\end{equation}
From this, it follows immediately that
\begin{equation*}
    K(7^{\alpha}n+7^{\alpha}-2) \equiv 0 \pmod{7^{\alpha}},
\end{equation*}
thereby proving \eqref{eqconjdd}. The proof of \eqref{K1} can also be carried out by induction on $\alpha$, and is therefore omitted.

\section{Concluding Remarks}\label{sec:conc}

\begin{enumerate}
\item Corollary \ref{cor 8n+7} suggests that Corollary \ref{P0} can be strengthened for some cases. We leave this investigation to the interested reader.
\item It is desirable to have completely elementary proofs of Theorems \ref{abradu} and \ref{thmconjd}.
\item The following result is also true, however we skip the proof for the sake of brevity.
\begin{theorem}
    For all $n\geq 0$, we have
    \begin{align}
            P^\ast(9n+2)&\equiv 0\pmod 2,\label{P(9n+2)}\\ 
    P^\ast(9n+5)&\equiv 0\pmod 2,\label{P(9n+5)}\\ 
    P^\ast(10n+1)&\equiv 0\pmod 4,\label{P(10n+1)}\\ 
    P^\ast(10n+7)&\equiv 0\pmod 4,\label{P(10n+7)}\\ 
    P^\ast(10n+3)&\equiv 0\pmod 8,\label{P(10n+3)}\\ 
    P^\ast(10n+5)&\equiv 0\pmod 8.\label{P(10n+5)}
    \end{align}
\end{theorem}
\end{enumerate}

\section*{Declarations}

\begin{itemize}
\item Funding: The work of the second author is partially supported by an Ahmedabad University Start-Up Grant (Reference No. URBSASI24A5).
\item Conflict of interest/Competing interests: The authors declare no conflicts of interests or competing interests.
\item Ethics approval and consent to participate: Not applicable.
\item Consent for publication: Not applicable.
\item Data availability: Not applicable.
\item Materials availability: Not applicable.
\item Code availability: Made available via URLs linked in the text.
\item Author contribution: All authors contributed equally.
\end{itemize}

\end{document}